\documentclass[]{amsart}

\usepackage{amssymb}
\usepackage{amsthm}
\usepackage{thmtools}
\usepackage{graphicx}
\usepackage{bm}
\usepackage{microtype}
\usepackage{dsfont}
\usepackage{mathtools}
\usepackage{tikz}
\usepackage[
  backend=biber,
  bibencoding=utf-8,
  doi=true,
  url=false
]{biblatex}
\usepackage{comment}
\usepackage[hidelinks]{hyperref}
\usepackage[capitalize]{cleveref}
\usepackage{booktabs,multirow}
\usepackage{subcaption}
\usepackage[ruled,vlined]{algorithm2e}
\usepackage{etoolbox}
\usepackage{csquotes}
\usepackage{nicefrac}
\usepackage{orcidlink}

\usepackage[normalem]{ulem}

\usetikzlibrary{decorations.markings}

\newcounter{mythm}
\crefalias{mythm}{section}
\numberwithin{mythm}{section}

\declaretheorem[style=plain, sibling=mythm]{theorem}
\declaretheorem[style=plain, sibling=mythm]{lemma}

\declaretheorem[style=plain, sibling=mythm]{corollary}
\declaretheorem[style=definition, sibling=mythm]{definition}
\declaretheorem[style=definition, sibling=mythm]{example}
\declaretheorem[style=remark, sibling=mythm]{remark}

\DeclareMathOperator{\Vol}{Vol}
\DeclareMathOperator{\diam}{diam}

\newcommand{\pluseq}{\mathrel{+}=}

\renewcommand{\d}{\,\mathrm{d}}
\newcommand{\sums}{\sideset{}{'} \sum}

\renewcommand{\Re}[1]{\mathrm{Re}(#1)}
\renewcommand{\Im}[1]{\mathrm{Im}(#1)}

\NewBibliographyString{refname}
\NewBibliographyString{refsname}
\DefineBibliographyStrings{english}{%
  refname = {ref\adddot},
  refsname = {refs\adddot}
}

\newcommand{\new}[1]{#1}

\begin{document}
\title[Computation of dipolar interactions with
PBC]{Zeta Expansion for Long-Range Interactions under Periodic Boundary Conditions with Applications to Micromagnetics}

\author[Buchheit]{Andreas A. Buchheit$^{1,2}$\orcidlink{0000-0003-4004-713X}}
\address{\textsuperscript{1} Department of Mathematics, Saarland University, 66123 Saarbrücken, Germany}
\address{\textsuperscript{2} Department of Mathematics, ETH Zürich, Rämistrasse 101, Zürich, 8092, Switzerland
}

\author[Busse]{Jonathan K. Busse$^{1,3}$\orcidlink{0009-0001-3323-3455}}
\address{\textsuperscript{3} High-Performance Computing Department, German Aerospace Center (DLR), 51147 Cologne, Germany}

\author[Keßler]{Torsten Keßler\textsuperscript{4}\orcidlink{0000-0003-2530-6746}}
\address{\textsuperscript{4} Department of Mechanical Engineering, Eindhoven University of Technology, 5600
MB Eindhoven, Netherlands}

\author[Rybakov]{Filipp N. Rybakov\textsuperscript{5}\orcidlink{0000-0002-3577-7966}}
\address{\textsuperscript{5} Department of Physics and Astronomy, Uppsala University, Box-516, Uppsala SE-751 20, Sweden}

\begin{abstract}
We address the efficient computation of power-law-based interaction potentials of homogeneous $d$-dimensional bodies with an infinite $n$-dimensional array of copies, including their higher-order derivatives. This problem forms a serious challenge in micromagnetics with periodic boundary conditions and related fields. 
Nowadays, it is common practice to truncate the associated infinite lattice sum to a finite number of images, introducing uncontrolled errors. We show that, for general interacting geometries, the exact infinite sum for both dipolar interactions and generalized Riesz power-law potentials can be obtained by complementing a small direct sum by a correction term that involves efficiently computable derivatives of generalized zeta functions. We show that the resulting representation converges exponentially in the derivative order, reaching machine precision at a computational cost no greater than that of truncated summation schemes. In order to compute the generalized zeta functions efficiently, we provide a superexponentially convergent algorithm for their evaluation, as well as for all required special functions, such as incomplete Bessel functions. Magnetic fields and related quantities can thus be evaluated to machine precision in arbitrary cuboidal domains periodically extended along one or two dimensions. We benchmark our method against known formulas for magnetic interactions and against direct summation for Riesz potentials with sufficiently large exponents, consistently achieving full precision. In addition, we identify new corrections to the asymptotic limit of the demagnetization field and tabulate high-precision benchmark values that can be used as a reliable reference for micromagnetic solvers. The techniques developed are broadly applicable, with direct impact in other areas such as molecular dynamics.
\end{abstract}

\keywords{Epstein zeta function, micromagnetics, numerical analysis,  lattice sums, meromorphic continuation, condensed matter physics, number theory}

\maketitle

\begin{figure}[h!]
    \centering
\includegraphics[width=0.7\linewidth]{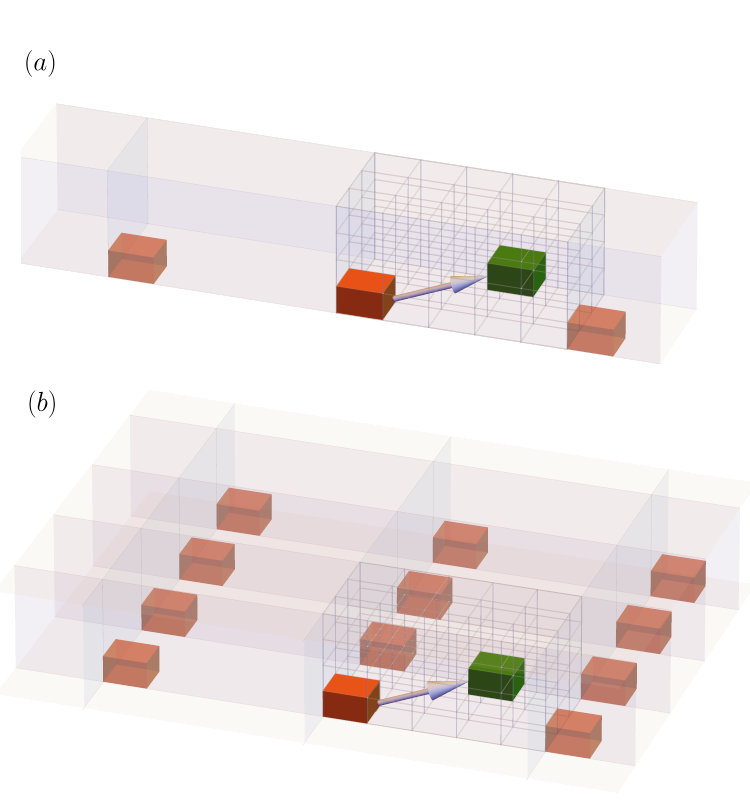}
    \caption{
    Sketch illustrating a \new{3D system 
    of} $N_x\times N_y\times N_z=5\times 3\times 4$
    interacting cuboids repeated
    infinitely in $x$ direction (a), and both in $x$ and $y$ direction (b).
    The green cuboid is influenced by the red cuboid indicated by the arrow, as well as by the total influence of all opaque red cuboids
    }
    \label{fig:boxes}
\end{figure}

\section{Introduction and problem description}
Microscopic textures of magnetization in magnetic materials are extremely diverse, interesting for applications, and challenging to observe and explain. 
Numerous studies with a century-long history are devoted to the theory of micromagnetic states~\cite{LL1935, Stewart1954, KITTEL1956437, ArrottGoldman1959, Brown1962, CraikTebble1965, MalozemoffSlonczewski_1979, Hubert, DeSimone_chapter4_2006}. 
The ``micromagnetics'' formulated back in 1959~\cite{micromagnetics_Brown, micromagnetics_Aharoni, Brown1963} is a powerful theoretical approach to understanding these phenomena.
Approaches to solving micromagnetic problems can be divided into two large classes: with the explicit involvement of magnetic field potential (either scalar or vector) or without it~\cite{Asselin_Thiele_1986, MicromagneticsRevisited}.
Naturally, both approaches have their advantages and disadvantages.

The well-known advantages of potential-based approaches include those provided by finite element methods. 
Namely, the ability to accurately simulate samples of complex shapes by representing them as composed of simplices such as tetrahedrons.
At the same time, it is challenging to accurately account for fields outside the sample -- semi-infinite regions where the decay has an algebraic nature. 
There is a large literature devoted to the relevant methods; for finite elements in micromagnetics, see Refs.~\cite{Schrefl_FEM_2007, Hertel_FEM_2007, ABERT201329, EXL2019179}, and for exterior issues, e.g. Neumann-to-Dirichlet map, see Refs.~\cite{IngermanDruskinKnizhnerman, MURATOV2006637, Druskin_SIAM_Review_2016}.
However, there is an important and interesting feature here.
In those dimensions where periodic boundary conditions (PBC) are assumed, the non-trivial component of the potential field is naturally a periodic function, and, accordingly, the problem of semi-infinities disappears along the corresponding dimensions.

On the other hand, potential-free approaches are well applicable for rectangular geometries. 
One strong advantage is the translation invariance when discretized by identical cuboids (rectangular prisms). 
The calculation of the interaction of ``everyone with everyone'' is accelerated by applying the fast Fourier transform (FFT) and the convolution theorem for zero-padded arrays~\cite{Yuan_1992, Hayashi_1996, Hubert}. 
Convolution kernels can be precomputed once for a given geometry and used at each iteration as constant arrays. 
However, in contrast to potential-based methods, the challenge here is precisely the case of PBC, see Fig.~\ref{fig:boxes}. 
This is because the convolution kernel requires calculating sums of infinite and slowly converging series for each of the dimensions with PBC, see~\cite{Lebecki_2008}.
It is important to emphasize that in modern studies, these sums are approximated by a naive cutoff technique. 
For example, in popular software packages such as MuMax3~\cite{MuMax3}, the corresponding parameter, {\sffamily \mbox{SetPBC(int, int, int)}}, is the number of ``repeats'', and due to the related time consumption, these integers rarely exceed the value of~10, see e.g. the Supplemental Material in~\cite{PhysRevLett.118.087202} and Refs.~\cite{Pierobon2018, PhysRevB.111.184423, Treves2025, JEFREMOVAS2025100036}. 
Some other relevant software packages are as follows: OOMMF~\cite{OOMMF}, BORIS~\cite{Boris_software}, magnum.np~\cite{Bruckner2023}, Ubermag~\cite{Ubermag_software}, \new{Julia~\cite{Wang_2024}}.
The above-mentioned sums cannot be reliably approximated by classic techniques such as Ewald summation~\cite{ewald1921berechnung}, although in some special cases, such as when the cuboids are cubes, there is numerical evidence of significant improvement in accuracy~\cite{Lebecki_2008}.
In general, the problem exists because the interaction potentials are functions with rather complex profiles such as compositions of power, trigonometric and hyperbolic functions, which are further complicated by dependencies on the aspect ratios of the cuboid (see Refs.~\cite{RhodesRowlands1954, Hubert} and \ref{app:cuboids_classical}).
As a result, truncating the series and/or replacing functions with simpler ones when calculating dipole-dipole interactions leads to an inconsistency of the boundary conditions in the sense that the periodicity of local interactions is exact, but that of non-local interactions is inaccurate. 
This causes discrepancies that can no longer be reduced by mesh refinement. \new{We note that efficient Ewald-type schemes that treat the infinite sums for point charges for quasi two-dimensional geometries have been developed for Coulomb interactions in \cite{gan2025fast,gan2025random}. However, these methods are not directly applicable in the micromagnetic settings, where one must evaluate more general interaction potentials between finite-volume domains together with their derivatives.}

The present work solves this problem. Some of the authors have recently generalized the classic Euler--Maclaurin summation formula 
to include power-law long-range interactions on higher-dimensional lattices \cite{buchheit2022efficient,buchheit2025computation,buchheit2023exact}, leading to the singular Euler-Maclaurin (SEM) expansion. This method \new{allows one to exactly describe} discrete lattices in terms of their continuous analogs with an efficiently computable correction term \cite{buchheit2023exact}.  In a recent major advancement, we have extended this method to lattice subsets with boundaries \cite{buchheit2025computation}. These methods are based on efficiently computable generalized zeta functions \cite{buchheit2024epstein} and allow for the computation of non-trivial high-dimensional lattice sums and their meromorphic continuations to full precision. In this work, we further generalize this framework to the use case appearing in micromagnetics.
We show how to compute interaction energies in domains $\Omega\subset\mathds R^d$, typically three-dimensional cuboids, with infinite arrays of periodic images, arising from general Riesz-type long-range interactions of the form $|\bm{r}|^{-\nu}$, including their meromorphic continuation in $\nu$ efficiently to full precision. 
Importantly, our framework supports the evaluation of interaction energy derivatives of arbitrary order. By choosing derivatives up to order~2 and $\nu = 1$, we recover all quantities required for the simulation of micromagnetic systems in $d$~dimensions, periodically repeated along an $n$-dimensional subarray with ${n < d}$. 
Note that for micromagnetic problems, the natural dimension of space is $d=3$, and lattices are typically one-dimensional $(n=1)$ or two-dimensional $(n=2)$, which corresponds to the geometry of wires, Fig.~\ref{fig:boxes}(a), or films, Fig.~\ref{fig:boxes}(b), respectively.
The method, however, extends to general interactions and geometries.

We begin our exposition by introducing the generalized interaction potential of a body $\Omega$ with an array of copies, each centered at a point in a set $L\subseteq \mathds R^d$.
\begin{definition}[Generalized potential with periodic boundary conditions]
Let $\Omega$ be a compact domain in $\mathds R^d$. We define the potential function of a source object $\Omega$ located at position $\bm{r}\in\mathds R^d$, interacting with an $n$-dimensional arbitrary Bravais lattice $L$ in $d$-dimensional space of identical objects.
For an interaction exponent $\nu > n$, and a multi-index of derivative orders $\bm m \in \mathbb{N}^d$, the potential reads
\begin{equation}
   U^{(\bm m)}(\bm r)= \sum_{\bm z\in L}  \int_{\Omega+\bm z} \nabla_{\bm r}^{\bm m} \int_{\Omega+\bm r} \vert \bm r'-\bm r''\vert^{-\nu}\,\mathrm d \bm r'\,\mathrm d \bm r''.
\label{U}
\end{equation}
The potential can be meromorphically continued to $\nu \in \mathds C$. 
\end{definition}

The above potential function allows us to evaluate all potentials and magnetic fields appearing in micromagnetics. While highly efficient methods exist for evaluating the double integral for single summands, see \ref{app:cuboids_classical}, the evaluation of the infinite and slowly decaying lattice sum has posed severe computational problems so far, as we described above.
Fast methods for kernel evaluation like the Fast Multipole Method~\cite{Greengard1997,Cheng1999} and related methods
like Adaptive Cross Approximation~\cite{Bebendorf2000,BebendorfRjasanow2003} or polynomial interpolation~\cite[Chapter~4]{Boerm2010} speed up the nested
integrals over the interaction potential as they separate the integral into two independent ones.
But they do not accelerate the convergence of the infinite sum as the leading terms of the
corresponding expansions only decay algebraically
and hence do not improve the slow convergence of the infinite sum.
In contrast, here we leverage efficiently computable generalizations of the Riemann zeta function to compute the infinite lattice sum at a negligible additional cost in runtime compared to standard methods that truncate the real space sum. This permits the precise and efficient study of periodic magnetic textures with long-range interactions with vast application potential.

This work is structured as follows. In Section \ref{sec:main_result}, we present our efficiently computable representation of the generalized micromagnetics potential. In Section~\ref{sec:num_results}, we subsequently compare our method against direct summation for large exponents and against known results, obtaining full precision across the complete parameter range. \new{We further obtain} new asymptotic corrections to known formulas for the demagnetization field and provide high-precision reference values for future implementations in micromagnetics packages. In Section~\ref{sec:generalized_zeta_computation}, we discuss how to efficiently evaluate all special functions, which are required for our algorithm. In particular, we discuss how to evaluate the arising incomplete Bessel functions to full precision. We draw our conclusions and briefly describe how the results of this work can be applied to other problems in condensed matter physics and other disciplines
in Section~\ref{sec:outlook}.  

\section{Main result}
\label{sec:main_result}

In this work, we solve the problem of efficiently computing the infinite sums that arise in micromagnetic simulations to full precision. Our approach divides the sum into two parts, a small direct sum, already implemented in standard packages such as MuMax3, and a correction term that eliminates the truncation error. We then show that this correction term can be expressed in terms of derivatives of a generalized zeta function. By presenting an efficient algorithm for computing these derivatives to full precision, we enable the evaluation of the infinite lattice sum at essentially the same numerical cost as the truncated sum.

We will need Fourier transforms in the following, for which we choose the following convention.
\begin{definition}[Fourier transformation]
For an integrable function $f : \mathds R^d \to \mathds C$ the Fourier transformation  $\mathcal F: f \mapsto \hat f$
is given by
\begin{equation*}
\hat f(\bm \xi) = \int_{\mathds R^d} f(\bm x) e^{-2 \pi i \bm x \cdot \bm \xi} \, \text d \bm x,
\quad \bm \xi \in \mathds R^d.
\end{equation*}
\end{definition}

We subsequently define the zeta function for a point set $L\subseteq \mathds R^d$, shifted by a vector  $\bm r\in \mathds R^d$.
\begin{definition}[Set zeta function]
Let $L\subseteq \mathds R^d$ be a uniformly discrete set, $\bm r\in \mathds R^d$, and $\nu\in \mathds C$ with $\mathrm{Re}(\nu)>d$. We then define the  zeta function of the set  $L\backslash\{\bm r\}$ as
the meromorphic continuation of
\begin{equation*}
Z_{L,\nu}(\bm r)=\sideset{}{'}\sum_{\bm z\in L}\frac{1}{\vert \bm z-\bm r \vert^\nu},
\end{equation*}
to $\nu\in\mathds C$,
where the primed sum excludes the summand where $\bm z=\bm r$. 
\end{definition}
Note that in contrast to \cite{buchheit2025computation}, the argument of the zeta function here refers to the shift in position. The resulting zeta function is related to the Epstein zeta function
originally introduced by Paul Epstein in 1903 in
\cite{epstein1903theorieI,epstein1903theorieII}.
\new{For some} lattice $\Lambda=A\mathds Z^{d}$, $A\in\mathds{R}^{d\times d}$ regular, a \new{shift vector} $\bm r\in\mathds R^d$, and a \new{wave vector} $\bm k\in\mathds R^d$, the Epstein zeta function is defined as the lattice sum
$$
Z_{\Lambda,\nu}(\bm r,\bm k)=
\sums_{\bm z\in \Lambda}\frac{e^{-2\pi i \bm k\cdot\bm z}}{|\bm z-\bm r|^{\nu}},\qquad \Re\nu>d,
$$
meromorphically continued to $\nu\in\mathds C$. 
The Epstein zeta function is a generalization of the Riemann zeta function to oscillatory lattice sums in multiple dimensions and is of particular interest in the physics of long-range interacting lattice systems \cite{buchheitEpsteinZetaMethod2025}.
The set zeta function was recently introduced by some of the authors in \cite{buchheit2025computation} as a generalization to geometries without translational invariance.

The following theorem is our first main result. It provides a way to correct the error introduced by truncating the lattice sum through efficiently computable derivatives of zeta functions, yielding an exact expression for the full sum. We refer to this method as \emph{zeta expansion}. To present the result in a compact form, we adopt standard multi-index notation, i.e., $\bm \alpha! = \alpha_1! \alpha_2! \dots \alpha_d!$. The proof relies on the fact that the set zeta function is holomorphic in $\bm r$ on a subset of $\mathds C^d$, see Corollary~\ref{cor:holomorphy}, and can be expanded into a Taylor series that converges uniformly on every ball that does not intersect $L$. This Corollary follows directly from Theorem~\ref{thm:crandall} that also yields a computationally efficient representation of the infinite sum.

\begin{theorem}[Zeta expansion for micromagnetics]
\label{thm:sem}
Consider $L=A\mathds Z^n\times \{0\}^{d-n}$ with $A\in \mathds R^{n\times n}$ regular and $n\in \mathds N_+$ with $n\le d$, a multi-index of derivative orders $\bm m\in \mathds N^d$, and a compact domain $\Omega\subset \mathds R^d$. 
For $\bm r\in \mathds R^d$, choose a symmetric near-field lattice $L_\mathrm{near}\subset L$ such that $-L_\mathrm{near}=L_\mathrm{near}$ and a far-field lattice $L_\mathrm{far}=L\setminus L_\mathrm{near}$ such that $\mathrm{dist}_\infty(\bm r, L_\mathrm{far}+2\Omega)>0$ with $\mathrm{dist}_\infty$ the distance in the infinity norm. 

The potential admits the near/far-field decomposition,
\[
    U^{(\bm m)}(\bm r) = \sum_{\bm z \in L_\mathrm{near}}   S^{(\bm m)}(\bm r+\bm z) + \mathcal Z_{L_\mathrm{far}}^{(\bm m)}(\bm r)
\]
with a truncated sum over $S^{(\bm m)}(\bm{r})$,
\[
S^{(\bm m)}(\bm r)=\int_{\Omega} \nabla_{\bm r}^{\bm m} \int_{\Omega+\bm r} \vert \bm r'-\bm r''\vert^{-\nu}\,\mathrm d \bm r'\,\mathrm d \bm r''.
\]
The correction term $\mathcal{Z}_{L_\mathrm{far}}^{(\bm m)}(\bm{r})$ can be expressed in terms of a zeta function acted on by a shape-dependent differential operator $\mathcal D_\Omega$,
\[
\mathcal Z_{L_\mathrm{far}}^{(\bm m)}(\bm r)= \mathcal D_\Omega Z_{L_\mathrm{far},\nu}^{(\bm m)}(\bm r).
\]
For a general compact domain $\Omega$, the operator admits the convergent Taylor representation
\[
\mathcal D_\Omega 
= \sum_{\bm \alpha\ge \bm 0} \frac{1}{\bm \alpha!}\, 
\mathcal I_{\bm \alpha}(\Omega)\,\nabla^{\bm \alpha}, 
\qquad
\mathcal I_{\bm \alpha}(\Omega) 
= \iint_{\Omega\times\Omega}(\bm r-\bm r')^{\bm \alpha}\,\mathrm d \bm r\,\mathrm d \bm r',
\]
where, if $\Omega$ is centrally symmetric, $\Omega=-\Omega$, only multi-indices $\bm \alpha$ with all components even contribute, since otherwise $\mathcal I_{\bm \alpha}(\Omega)=0$.

For the axis-aligned cuboid  $\Omega=\prod_{i=1}^d\left(-c_i/2,c_i/2\right)$,  $c_j>0$, the operator $\mathcal D_\Omega$ takes the explicit form
\[
\mathcal D_\Omega
= \sum_{\bm \alpha \ge \bm 0} 
\frac{2^d}{\bigl(2(\bm \alpha+\bm 1)\bigr)!}\,
\bm c^{\,2(\bm \alpha+\bm 1)}\,\nabla^{2\bm \alpha},
\qquad \bm c=(c_1,\dots,c_d)^T.
\]
with $\bm 1=(1,\ldots,1)^T$.
\end{theorem}
\begin{proof}
    We begin by removing the sum over $L_\mathrm{near}$. In the remaining sum over $L_\mathrm{far}$, notice that the derivative with respect to $\bm r'$ can be interchanged with the derivative with respect to $\bm r$ due to linearity. By Fubini's theorem, we can exchange the order of integration and differentiation, leading to
    \[
    \mathcal Z_{L_\mathrm{far}}^{(\bm m)}(\bm r)= \iint_{\Omega\times\Omega} \new{Z_{L_\mathrm{far},\nu}^{(\bm m)}(\bm r+\bm r'-\bm r'')\,\mathrm d\bm r'\,\mathrm d \bm r'',}
    \]
    \new{recalling that \[
    Z_{L_\mathrm{far},\nu}(\bm r)=\sideset{}{'}\sum_{\bm z\in L_\mathrm{far}}\frac{1}{\vert \bm z-\bm r \vert^\nu}.
    \]}
    Now, by holomorphy of the set zeta function from Corollary~\ref{cor:holomorphy}, and by the restriction on the distance of $\bm r$ to $L_\mathrm{far}$, the integrand and all its derivatives  can be expanded in uniformly convergent Taylor series within this domain, yielding
    \[
    \mathcal Z_{L_\mathrm{far}}^{(\bm m)}(\bm r)=\iint_{\Omega\times\Omega} \sum_{\bm \alpha\ge \bm 0}\frac{(\new{\bm r'-\bm r''})^{\bm \alpha} }{\bm \alpha!} Z_{L_\mathrm{far},\nu}^{(\bm m+\bm \alpha)}(\new{\bm r})\,\new{\mathrm d\bm r'\,\mathrm d \bm r''}=\sum_{\bm \alpha\ge \bm 0}\frac{\mathcal I_{\bm \alpha}(\Omega)}{\bm \alpha!} Z_{L_\mathrm{far},\nu}^{(\bm m+\bm \alpha)}(\new{\bm r}),
    \]
    where the order of integration and summation was exchanged due to Fubini. For the case of the cuboid, the moment integral factorizes component-wise, and we have
    \[
        \int_{-c_i/2}^{c_i/2} \int_{-c_i/2}^{c_i/2} \new{(r_i'-r_i'')^{\alpha_i}\,\mathrm d r_i'\,\mathrm d r_i''} = \frac{2 c_i^{\alpha_i+2}}{(\alpha_i+1)(\alpha_i+2)},\quad \text{all}~\alpha_i~\text{even},
     \]
     and zero otherwise. Recombining $\mathcal I_{\bm \alpha}(\Omega)$ then yields the specific form of the operator for the cuboid. 
\end{proof}

\new{The results of the preceding theorem apply to arbitrary source and target geometries, including domains with curved boundaries, since the geometry enters exclusively through the moment integral $\mathcal I_{\bm \alpha}(\Omega)$. To illustrate this, we here provide the example of interacting $d$-dimensional balls.}
\new{\begin{example}
By the multinomial theorem, the moment integral can be expressed directly in terms of the moments of $\Omega$,
\[
\mathcal I_{\bm \alpha}(\Omega) = \sum_{\bm \beta\le \bm \alpha} \binom{\bm \alpha}{\bm \beta} (-1)^{-\vert \bm \beta\vert} M_{\bm \beta}(\Omega)M_{\bm \alpha-\bm \beta}(\Omega),\quad M_{\bm \beta}(\Omega) = \int_{\Omega} \bm r^{\bm \beta}\,\mathrm d \bm r.
\]
Consider now $\Omega = B_R$ the $d$-dimensional ball of radius $R$, centered at zero. Then $M_{\bm \beta}(B_R)$ vanishes if $\bm \beta$ includes any odd entries as $-\Omega=\Omega$. The even moment integrals admit the analytic form \cite[6.2.1, 6.2.2]{abramowitz1948handbook}
\[
M_{2\bm{\beta}}(B_R)
= R^{d+2|\bm{\beta}|}
\frac{\prod_{i=1}^d \Gamma\!\left(\beta_i + \tfrac{1}{2}\right)}
{\Gamma\!\left(|\bm{\beta}| + \tfrac{d}{2} + 1\right)},\quad \bm \beta \in \mathds N_0.
\]
\end{example}
}

By enlarging the real space sum, we increase the radius of analyticity of the zeta function $Z_{L_\mathrm{far},\nu}(\bm r)$ around $\bm 0$. 
As a result, the larger the real space sum and the smaller the size of $\Omega$, the fewer derivatives that must be included to reach full precision. 
These derivatives of the zeta function can then be written in terms of a truncated sum plus two superexponentially convergent sums, all of which can be efficiently evaluated. 

The error in truncating the number of derivatives decreases exponentially in the number of \new{derivatives} and algebraically in the components of $\bm c$. The precise error scaling is described in the following corollary.

\begin{corollary}[Error scaling]
\label{cor:error}
Assume that the
conditions of \Cref{thm:sem}
hold. Fix $\eta>1$, such that 
$$\operatorname{dist}_\infty(\bm 0,L_{\rm{far}}+2\Omega)>\eta\;\! d\diam(\Omega),$$ with
$\diam(\Omega)=\sup_{\bm r,\bm r'\in \Omega} \Vert \bm r- \bm r' \Vert_{\infty}$
 the diameter of $\Omega$. For any $\bm r\in \mathds R^d$ such that \[
\|\bm r\|_{\infty}<\eta\;\! d\diam(\Omega),
\]
and $\ell\in\mathds N$, the remainder when truncating to derivatives of order $|\bm\alpha|\le \ell$,
$$
R_\ell
= \sum_{|\bm \alpha|>\ell} \frac{1}{\bm \alpha!}\, 
\mathcal I_{\bm \alpha}(\Omega)\,\nabla^{\bm \alpha} Z_{L_\mathrm{far},\nu}^{(\bm m)}(\bm r),
$$
can be uniformly bounded in $\bm r$ by
\[
|R_\ell| \le  M
\big(\Vol(\Omega)\big)^2
\frac{\eta^{-\ell}}{(\eta-1)}
\]
where
$
M=
\sup\{Z^{(\bm m)}_{L_\mathrm{far},\nu}(\bm z):
\bm z\in\mathds C^d,\ \|\bm z\|_{\infty}<\eta\;\! d\diam(\Omega)\}
$.
\end{corollary}

\begin{proof}
By Corollary~\ref{cor:holomorphy}, $Z_{L_{\rm{far}}}^{(\bm m)}(\bm \cdot)$
is holomorphic on the polydisc
$$
\mathds D=\{\bm z\in\mathds C^d:\|\bm z\|_{\infty}< \rho\}
$$
with radius $\rho = \eta\;\! d\diam(\Omega)$.
From Cauchy's estimate, it follows that
\[
\sup_{\bm z\in \mathds D}|\nabla^{\bm \alpha} Z^{(\bm m)}_{L_\mathrm{far},\nu}(\bm z)|\le M\frac{\bm \alpha!}{\rho^{|\bm \alpha|}}.
\]
Using further that
\[
I_{\bm\alpha}(\Omega)=\iint_{\Omega\times\Omega}(\bm r-\bm r')^{\bm \alpha}\,\mathrm d \bm r\,\mathrm d \bm r'\le (\diam(\Omega))^{|\bm\alpha|}(\Vol(\Omega))^2
\]
we can bound the remainder by
$$
|R_{\ell}|
\le M
(\Vol(\Omega))^2
\sum_{|\bm \alpha|>\ell} 
(\eta d)^{-|\bm\alpha|.
}
$$
After setting $\bm q=\bm 1/(\eta d)$ and thus $(\eta d)^{-\vert \bm \alpha\vert}= \bm q^{\bm \alpha}$, the multinomial theorem yields for $k\in \mathds N$,
$$
\sum_{|\bm \alpha| = k}
\bm q^{\bm\alpha}
\le
\sum_{|\bm \alpha| = k}
\frac{k!}{\bm\alpha!}
\bm q^{\bm\alpha}
=(q_1+\dots+q_d)^k=(d/(\eta d))^{k}=\eta^{-k}.
$$
After evaluating the geometric series,
\[
\sum_{k=\ell+1}^\infty \eta^{-k}=
\frac{\eta^{-\ell}}{(\eta-1)},
\]
we obtain the asserted bound.
\end{proof}

As our second main result, we present a representation of the set zeta function and its derivatives in terms of superexponentially decaying sums. The resulting representation serves as the basis for its efficient computation and can be used to derive the properties of the generalized zeta function. We name this representation Crandall representation due to the close connection to the exponentially convergent representation of the Epstein zeta function derived by Crandall in \cite{crandall2012unified}.  The underlying method can be considered as a generalization of Ewald summation \cite{ewald1921berechnung}.
\begin{theorem}[Crandall representation of set zeta functions]
\label{thm:crandall} Let $L=A\mathds Z^n\times \{0\}^{d-n}$, $n\in \mathds N_+$ with $n\le d$, and $\bm \alpha\in \mathds N^d$. 
For $\bm r\in \mathds R^d$, choose a finite near-field lattice $L_\mathrm{near}\subset L$ and set $L_\mathrm{far}=L\setminus L_\mathrm{near}$. Adopt the notation $\bm r^{\parallel}=(r_1,\dots,r_{n})^T$ and $\bm r^{\perp}=(r_{n+1},\dots,r_{d})^T$ and likewise for $\bm \alpha$. For $\Re{\nu}>n$, derivatives of the zeta function of the far-field lattice then admit the representation
\begin{align*}
Z_{L_\mathrm{far},\nu}^{(\bm \alpha)}(\bm r)=
\frac{(\pi/\lambda^2)^{\nu/2}}{\Gamma(\nu/2)}\bigg[
&\sum_{\bm z\in L_\mathrm{far}}
 \lambda^{-\vert \bm\alpha\vert}
G_\nu^{(\bm \alpha)}\Big(\frac{\bm{z}-\bm{r}}{\lambda}\Big)
-\sum_{\bm z\in L_\mathrm{near}}\lambda^{-\vert \bm\alpha\vert} g_\nu^{(\bm \alpha)}\Big(\frac{\bm{z}-\bm{r}}{\lambda}\Big)
\\ &\quad+ \frac{\lambda^{n}}{V_\Lambda} \sum_{\bm k\in \Lambda^{\ast}} (-2\pi i \bm k)^{\bm \alpha^{\parallel}} e^{-2\pi i \bm k \cdot \bm r^{\parallel}} \lambda^{-\vert \bm \alpha^{\perp}\vert}G_{n-\nu}^{(\bm \alpha^{\perp})}\Big(\lambda \bm k, \frac{\bm r^{\perp}}{\lambda}\Big)\bigg].
\end{align*}
where $\Lambda^{\ast}=A^{-T}\mathds Z^n$. \new{Here,} $\lambda>0$ denotes the Riemann splitting parameter and the derivative always acts with respect to $\bm r$ (or $\bm r^\perp$). Furthermore, we define the upper Crandall function $G_\nu$ and the lower Crandall function $g_\nu$ as
\[
G_\nu(\bm z) = \frac{\Gamma(\nu/2,\pi \bm z^2)}{(\pi \bm z^2)^{\nu/2}}, \qquad g_\nu(\bm z) =\frac{\gamma(\nu/2,\pi \bm z^2)}{(\pi \bm z^2)^{\nu/2}},
\]
with $\Gamma(a,z)$ the upper and $\gamma(a,z)$ the lower incomplete Gamma function.
Here $g_{\nu}(\bm z)$ is continuously extended to $\bm z=\bm 0$ and we define $G_{\nu}(\bm 0)=-\nu/2$. Finally, $G_{\nu}(\bm k,\bm r)$ is an incomplete Bessel function, 
\[
G_{\nu}(\bm k,\bm r)=2\int_0^1 t^{-\nu} e^{-\pi \bm k^2/t^2} e^{-\pi \bm r^2 t^2}\, \frac{\mathrm d t}{t}.
\]
The representation forms the meromorphic continuation to $\nu\in\mathds C$.
\end{theorem}
\begin{proof}
We follow the proof of \cite[Th. 2.11]{buchheit2024epstein}
and begin by deriving the Crandall representation for $\bm\alpha=\bm 0$.
Let $\nu>d$, so 
the defining lattice sum of the set zeta function converges absolutely by  
\cite[Lemma A.1.]{buchheit2024epstein}.
Note that the lattice sum is no longer primed, as  \(\bm r\notin L_{\rm{far}}\).
By the fundamental relation for the Crandall functions
\cite[Lemma 2.9 (2)]{buchheit2024epstein} 
$$
( \bm z^2 )^{-\nu/2}\frac{\Gamma(\nu/2)}{(\pi/\lambda^2)^{\nu/2}} = G_\nu(\bm z/\lambda)+g_\nu(\bm z/\lambda)
$$
we may separate the set zeta function
into the upper and lower Crandall functions
with a Riemann splitting parameter $\lambda>0$. Further separating the sum involving the lower Crandall functions into a sum over the complete set $L$ minus a near-field sum yields
$$
Z_{L_{\rm{far}},\nu}(\bm r)\frac{\Gamma(\nu/2)}{(\pi/\lambda^2)^{\nu/2}} 
=
\sum_{\bm z \in L_{\rm{far}}} G_\nu\Big(\frac{\bm{z}-\bm{r}}{\lambda}\Big)
-\sum_{\bm z \in L_{\rm{near}}} g_\nu\Big(\frac{\bm{z}-\bm{r}}{\lambda}\Big)+
\sum_{\bm z \in L} g_\nu\Big(\frac{\bm{z}-\bm{r}}{\lambda}\Big).
$$
For the sum over the whole set $L=\Lambda\times \{0\}^{d-n}$, $\Lambda=A\mathds Z^n$,
we obtain
$$
\sum_{\bm z \in L} g_\nu\Big(\frac{\bm{z}-\bm{r}}{\lambda}\Big)
=
2\int_0^{1} \new{t^{-(\nu+1)}}
\Big(\sum_{{\bm z}^\parallel\in\Lambda}
e^{-\pi ({\bm z}^\parallel-\bm r^{\parallel})^2t^2/\lambda^2} \Big)e^{-\pi (\bm r^{\perp})^2 t^2/\lambda^2}\,\new{\mathrm d t},
$$
by the integral representation for the lower Crandall function in the proof of
\cite[Lem. 2.10]{buchheit2025computation},
where we exchanged the sum and the integral by dominant convergence for $\Re{\nu}>n$.
After applying Poisson summation 
\cite[Lem. 2.6]{buchheit2024epstein}
to the sum over $\Lambda$ with the Fourier transform
$$
\mathcal F\big(
e^{-\pi (\bm \cdot-\bm r^{\parallel})^2t^2/\lambda^2}\big)(\bm k)
=
\lambda^n t^{-n}e^{-2\pi i\bm k\cdot\bm r^{\parallel}}e^{-\pi \bm k^2 \lambda^2/t^2},
$$
we obtain
$$
\sum_{\bm z \in L} g_\nu\Big(\frac{\bm{z}-\bm{r}}{\lambda}\Big)
=
\frac{\lambda^{n}}{V_\Lambda} \sum_{\bm k\in \Lambda^{\ast}} e^{-2\pi i \bm k \cdot \bm r^{\parallel}}G_{n-\nu}\Big(\lambda \bm k, \frac{\bm r^{\perp}}{\lambda}\Big).
$$
Applying the derivative to the reassembled formula yields the statement.
\end{proof}

\new{At this point we briefly place the above representation in context with the existing literature on Ewald methods. First, Ewald-type representations equivalent to the standard Crandall formula for full point lattices and power-law potentials $|\bm{\cdot}|^{-\nu}$ are known, see, for instance, \cite[Eq.~226]{mazars2011long} for the case of a full three-dimensional lattice and references therein. The connection to the Crandall representation is obtained by selecting the Ewald cutoff in terms of incomplete Gamma functions, which can be shown to be equivalent to the splitting originally introduced by Riemann. One must be careful, however, because the resulting formulas are often not fully general, being restricted to special cases of $\nu$, $\Lambda$, or $d$. 
We also note important progress for quasi two-dimensional lattice sums in \cite[Eq.~234]{mazars2011long} and \cite[Eq.~28]{mazars2010ewald}, which employ incomplete Bessel functions, but are restricted to $\Lambda=\mathds{Z}^2$. Theorem~\ref{thm:crandall} generalizes these special cases to arbitrary $d$-dimensional lattices, complex exponents $\nu \in \mathds{C}$, and an arbitrary split into $n$ periodic dimensions and $d-n$ open-boundary dimensions. In addition, the representation allows for the removal of a finite number of singularities in real space without introducing cancellation errors. It enables the evaluation of higher-order derivatives, computable to high precision via Lemma~\ref{der-crandall}, where the appearing incomplete Bessel functions become efficiently computable as well by Algorithm~\ref{alg}. This provides the foundation for the zeta expansion, which uses derivatives of the regularized zeta function to efficiently handle arbitrary source and target geometries.
}

The holomorphic extension of the set zeta function, and thus its radius of analyticity, readily follows from the Crandall representation.
\begin{corollary}
\label{cor:holomorphy}
    Let the conditions of Theorem~\ref{thm:crandall} hold and define the set $D_{L_\mathrm{far}}\subseteq \mathds C^d$ as an intersection of $d$-dimensional complex cones centered at the point set $L_\mathrm{far}\subseteq \mathds R^d$,
$$
D_{L_\mathrm{far}}=\{\bm u\in\mathds C^d:|\Re{\bm u}-\bm z|>|\Im{\bm u}|\ \forall \bm z\in L_\mathrm{far}\}.
$$
Then $Z_{L_\mathrm{far},\nu}$ extends to a holomorphic function on $D_{L_\mathrm{far}}$.
\end{corollary}
\begin{proof}
    For ${\bm r}^\perp=\bm 0$, the function corresponds to an Epstein zeta function minus a finite sum over $L_\mathrm{near}$. Here, the holomorphy of $G_{\nu}(\bm \cdot)$ on
    $$D=\{\bm u\in\mathds C^d:|\Re{\bm u}|>|\Im{\bm u}|\}$$ and for $g_{\nu}(\bm \cdot)$ on $\mathds C^d$ has been shown in 
    \cite[Lemma 2.9]{buchheitComputationPropertiesEpstein2024} and the resulting holomorphy of the Epstein zeta function has been proven in \cite[Theorem 2.13]{buchheitComputationPropertiesEpstein2024}. This result then directly applies to the set zeta function. 
    
    For ${\bm r}^\perp \neq \bm 0$, notice from the integral representation of the incomplete Crandall function that for $\bm k\neq \bm  0$, $G_{n-\nu}(\bm k,{\bm r}^\perp)$ is entire in ${\bm r}^\perp$ with the bound.
    \[
    |G_{n-\nu}(\bm k,{\bm r}^\perp)|<|G_{n-\nu}(\bm k)|.
    \]
    For $\bm k=\bm 0$, the function reduces to a lower Crandall function, which is entire too. Thus, the set zeta function is holomorphic on $D_{L_\mathrm{near}}$ as the \new{compactly uniform limit} of holomorphic functions. 
\end{proof}

As the Crandall functions (and their derivatives) exhibit the asymptotic scaling
\[
|G_\nu(\bm k,\bm r)|\le |G_\nu(\bm r)|\sim \frac{e^{-\pi \bm r^2}}{\pi \bm r^2}, 
\]
see \cite{buchheit2025computation,buchheit2024epstein}, both sums fall off at a superexponential rate, building the foundation for an efficient computation of the zeta function. 

Note that all derivatives of the Crandall functions can be expressed analytically in terms of other Crandall functions with shifted exponents, as discussed in Sec.~\ref{sec:generalized_zeta_computation}. The complete algorithm for computing the set zeta function for arbitrary parameter choices is also presented there. Combining these results, Theorem~\ref{thm:sem} provides the exact correction of lattice sums due to truncation in terms of derivatives of generalized zeta functions, while Theorem~\ref{thm:crandall} establishes a superexponentially convergent representation of the required zeta functions and their derivatives. Together with the error scaling in Corollary~\ref{cor:error}, these results enable the efficient computation of the potential function and all its derivatives needed in micromagnetics.

\section{Numerical results}
\label{sec:num_results}

\subsection{Comparison against direct summation}

\begin{figure} 
    \centering
\includegraphics[width=.65\linewidth]{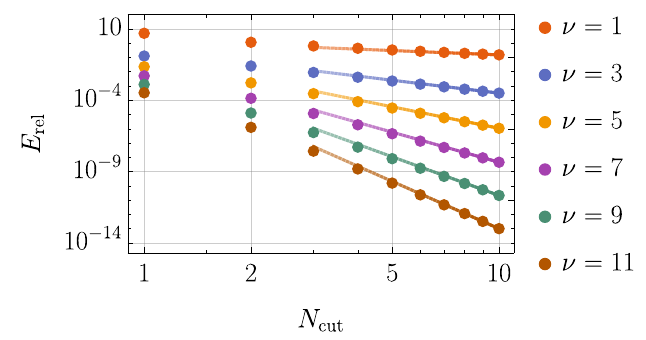}
    \caption{Relative error between the three-dimensional potential 
    $U^{(0,0,2)}(\bm r)$
    centered at $\bm r=(1,1,1)/2$ 
    for the lattice at $\mathds Z^2\times\{0\}$
    as obtained by the zeta representation in \Cref{thm:sem}
    in comparison with direct summation over the truncated grid 
    $\{-N_{\rm{cut}},\ldots,N_{\rm{cut}}\}^2\times\{0\}$ as dots a function of the 
    the truncation length integers
    $1\le N_{\rm{cut}}\le 10$.
    The error scaling $C_\nu N_{\rm{cut}}^{-\nu}$ for some fitted parameter $C_\nu$ is shown as lines of the in the same color as the corresponding the error dots.
    }
\label{fig:sum}
\end{figure}

Establishing reliable analytic benchmarks for non-trivial infinite lattice sums is a challenging task, as exact formulas are scarce, see, e.g., \cite{borwein2013lattice} and references therein. In the following, we use the fact that our method admits an arbitrary exponent $\nu$, and compare our method for sufficiently large $\nu$ with direct summation, which still yields high-precision reference values within reasonable times. 

To this end, we study the interaction potential
$U^{(\bm m)}(\bm r)$ for the cuboid domain
$\Omega=\prod_{i=1}^3[-c_i/2,c_i/2]$ centered at $\bm r$
for some $\bm c\in\mathds R^3_+$
with an infinite 2D array of copies centered at $L=\mathds Z^2\times \{0\}$.
Let $E_{\rm{rel}}$ denote the relative error
of our algorithm in comparison with high-precision values obtained by 
direct summation over the truncated lattice
$\{-N_{\rm{cut}},\ldots,N_{\rm{cut}}\}^2\times\{0\}$
for some $N_{\rm{cut}}\in\mathds N_+$.
Furthermore, we choose 
$L_{\rm{near}}=\{-1,0,1\}^2\times\{0\}$
in the near/far-field decomposition in 
\Cref{thm:sem} for the zeta treatment of the potential.
For benchmarking our implementation of the potential,
we take $\nu\ge 10$ and $N_{\rm{cut}}=14$.
We compare the potential obtained by adding the zeta-correction with up to fourth-order derivatives
to the $9$ summands in the grid $L_{\rm{near}}=\{-1,0,1\}^2\times\{0\}$ with the potential obtained by summing over
the truncated lattice
consisting of $841$ summands.
In \Cref{fig:directsum}
we show the error
for fixed 
second-order derivatives
$$\bm m\in\{(1,1,0),(1,0,1),(0,0,2)\}$$ 
and geometries
$$
\bm c\in\{(1,1,1)/50,(1,1,1)/100,(1,2,3)/100\}.
$$
over a grid of values 
$\bm r
=(1/4,1/4,r)
$, $r\in\{1/10,2/10,\ldots,5/10\}$ and $\nu\in [10,15]$ in increments of $1/10$. 
where the direct sum converges 
at least proportional to $N_{\rm{cut}}^{-10}$.
We observe, that our algorithm reaches full precision over the whole parameter range. As the error of our method does not exhibit a relevant scaling with $\nu$, see Corollary~\ref{cor:error}, a corresponding precision is to be expected for the whole $\nu$ range, which is confirmed by another analytic benchmark in the next section. 

 Assuming on this basis that the value of the zeta method is numerically exact, we now study the convergence of the direct sum against this reference for $\bm m=(0,0,2)^T$ and small values of $\nu$ as a function of the truncation $N_\mathrm{cut}$ in  \Cref{fig:sum}. We find that the relative error scales as $N_{\mathrm{cut}}^{-\nu}$, leading to slow convergence for small values of $\nu$. For the case of micromagnetics, $\nu=1$, and a typically used value of $N_\mathrm{cut}=10$, the relative error compared to the zeta method still exceeds $15\,\%$.

\begin{table}
\renewcommand{\arraystretch}{1.3}
\centering
\begin{tabular}{ccc}
    \toprule
    derivative vector $\bm m$
    & cuboid geometry $\bm c$ 
    & \(\max_{\bm r,\nu}E_{\mathrm{rel}}\) 
  \\

\midrule
\multirow{3}{*}{$(1,1,0)$}
& $(1,1,1)/50$
& $2.16 \cdot 10^{-16}$ 
\\

& $(1,1,1)/100$
& $2.12 \cdot 10^{-16}$ 
\\

& $(1,2,3)/100$
& $2.20 \cdot 10^{-16}$ 
\\

\midrule
\multirow{3}{*}{$(1,0,1)$}
& $(1,1,1)/50$
& $2.18 \cdot 10^{-16}$ 
\\

& $(1,1,1)/100$
& $2.11 \cdot 10^{-16}$ 
\\

& $(1,2,3)/100$
& $2.16 \cdot 10^{-16}$ 
\\

\midrule
\multirow{3}{*}{$(0,0,2)$}
& $(1,1,1)/50$
& $3.87 \cdot 10^{-16}$ 
\\

& $(1,1,1)/100$
& $2.14 \cdot 10^{-16}$ 
\\

& $(1,2,3)/100$
& $2.19 \cdot 10^{-16}$ 
\\
    \bottomrule
\end{tabular}
\caption{Relative error 
of the potential $U^{(\bm m)}(\bm r)$ of a cuboid
cuboid $\Omega=\prod_{i=1}^3[-c_i/2,c_i/2]$ centered at $\bm r$ and its images centered at lattice points $L=\mathds Z^2\times\{0\}$ for different second-order derivatives $\bm m$ and geometries $\bm c$. 
The value computed by our method with 
$N_{\rm cut}=1$ including up to fourth-order derivatives in the zeta correction term, is compared to direct summation with $N_{\rm cut}=14$ at sufficiently large values of $\nu$, where a reliable reference can still be obtained within reasonable times.
Each error is the maximum error obtained 
over a grid of values $\bm r=(1/4,1/4,r)$ for $1/10\le r\le 1/2$ and $10\le\nu\le 14$
in increments of $\Delta r=\Delta\nu=1/10$.
}
\label{fig:directsum}
\end{table}

\subsection{Corrections to asymptotic behavior of the demagnetization field}

As a second benchmark, we compare our numerical method against analytic expressions for the leading asymptotic behavior of the demagnetization factor.   To this end, we consider a three-dimensional cuboid $\Omega=[-c/2,c/2]^3$, centered at $\bm r= \bm 0$, and its infinite repetitions along the lattice $\Lambda = \mathds Z^2$, with the cuboid being magnetized along the $z$-axis. Dividing the elementary lattice cell into $N$ cuboids per dimension yields $c=1/N$, see \Cref{fig:demag}. 

\begin{figure}[h!]
    \centering
    \includegraphics[width=0.5\linewidth]{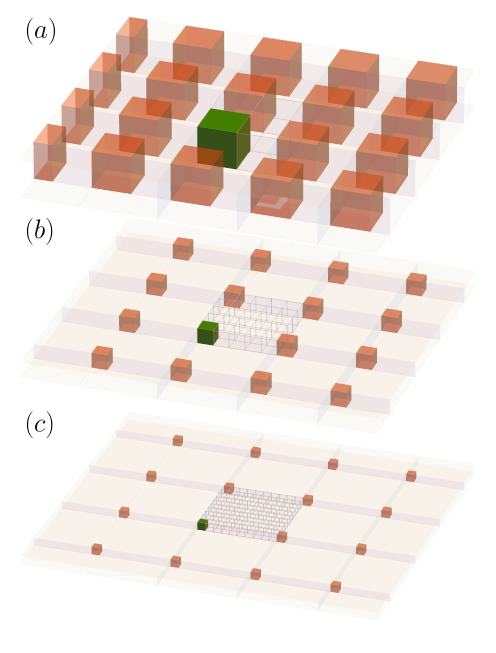}
    \caption{Cuboid 
    $\Omega=[-1/(2N),1/(2N)]^3$ and its infinite repetitions $\bm z+\Omega$
    along the two-dimensional lattice embedded in three dimensions
    $L=\mathds Z^2\times \{0\}$
    for $N=2$ (a), $N=5$ (b) and $N=10$ (c).
    }
    \label{fig:demag}
\end{figure}

The demagnetization factor $D_z$ then reads, 
\begin{align*}
D_z = 1 + \frac{1}{4\pi}\frac{1}{c^3}
\left(
U^{(2,0,0)}(\bm 0) + U^{(0,2,0)}(\bm 0)
\right),
\end{align*}
and thus reflects both the influence of the cuboid's own field (see \ref{app:cuboids_classical}) and the impact of an infinite number of neighbors.
We provide the values of $D_z$ beyond machine precision 
in 
\Cref{tab:dz}
.
One useful observation is that for $N=1$, the system develops into a continuous planar layer (of thickness~$c$), and therefore the demagnetization factor is trivial: 
\begin{align*}
D_z = 1 \quad (N=1).
\label{Dzleft}
\end{align*}
On the other hand, for $N\rightarrow\infty$, the cuboid is small compared to the lattice spacing, such that it interacts with its repetitions as a point dipole. The magnetic field contribution from the copies is expressed through a lattice sum, where the leading asymptotic corrections due to finite size of the cuboid can be analytically determined ~\cite{Borwein1987}, yielding 
\begin{align*}
D_z^{\mathrm{asym}}  = \frac{1}{3} + 
\frac{Z_{\mathds Z^2,3}(\bm 0)}{4\pi}\ \frac{1}{N^3},
\end{align*}
with the particular value of the Epstein zeta function \cite[Table 1, Sum $S(0,2)$]{zucker2017exact}
\[
Z_{\mathds Z^2,3}(\bm 0)
=
4\beta(3/2)\zeta(3/2),
\]
and where $\beta$ denotes the Dirichlet beta function.

In Fig.~\ref{fig:Dz_deviation}, we display the absolute difference between the demagnetization factor, obtained from the zeta expansion, and the known asymptotic formula as a function of $N$. We reproduce the known asymptotic behavior, providing another validation for our approach. In addition, we obtain the next order asymptotic correction,
\[
D_z=D_z^{\mathrm{asym}}- 
\alpha\ \frac{1}{N^7},\quad N\to \infty,
\]
where  $\alpha\approx 0.1441459732$ has been obtained from a fit. This correction results from the $\vert \bm \alpha\vert=1$ terms in the zeta expansion in Theorem~\ref{thm:sem}.

To serve as a reliable benchmark for future implementations of our algorithm in micromagnetics software packages, we provide multi-precision values of the demagnetization factor for various values of $N$ in Table~\ref{tab:dz}.

\begin{table}
\centering
\begin{tabular}{cc}
    \toprule
    $N$ 
    & \(D_z\) 
  \\

	\midrule
	  $2$
	& $0.4222\,0496\,3454\,0001\,7334\,0215\,3861\,9290$
	\\

	\midrule
	  $5$
	& $0.3390\,8248\,7690\,9804\,5966\,0589\,1398\,9491$
	\\

	\midrule
	  $10$
	& $0.3340\,5219\,1717\,7249\,3459\,0417\,7785\,8568$
	\\

	\midrule
	  $50$
	& $0.3333\,3908\,4315\,3283\,8944\,2139\,2339\,3160$
	\\

	\midrule
	  $100$
	& $0.3333\,3405\,2206\,1043\,3571\,9460\,7573\,5189$
	\\

    \bottomrule
\end{tabular}
\caption{
High-precision evaluation of
the demagnetization factor $D_z$ for different values of 
$N\in\{2,5,10,50,100\}$
for $32$ digits of precision.
For $N=2$ we chose $L_{\rm{near}}=\{-26,\ldots,-26\}^2\times\{0\}$
and derivatives up to the $10$'th order in the zeta correction term
and for 
$N\ge 5$ we chose $L_{\rm{near}}=\{-22,\ldots,-22\}^2\times\{0\}$
and derivatives up to the $8$th-order in the zeta correction term.
The zeta correction term was calculated over the grid $\{-6,\ldots,6\}^3$.
}
\label{tab:dz}
\end{table}

\begin{figure}[h!]
    \centering
    \includegraphics[width=0.5\linewidth]{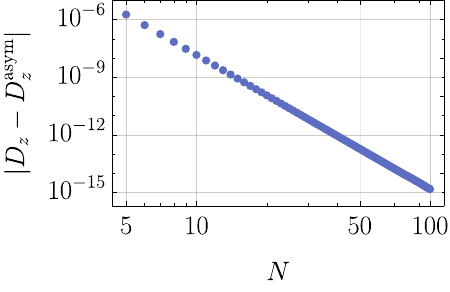}
    \caption{Difference between demagnetization factor $D_z$ for 3D cuboids  $\Omega=[-1/(2N),1/(2N)]^3$ interacting with a lattice $\Lambda = \mathds Z^2$, computed from our method, and the known asymptotic formula. The asymptotic behavior is correctly reproduced, as well as the corner cases ($D_z = 1$ for $N=1$ and $D_z \to 1/3$ for $N\to\infty$). The next order correction from the known asymptotics as $N\to \infty$ scales as $N^{-7}$ (obtained by a fit).}
    \label{fig:Dz_deviation}
\end{figure}

\section{Computation of generalized zeta functions}
\label{sec:generalized_zeta_computation}

In this section, we present the ingredients required for the efficient evaluation of the set zeta function. Theorem~\ref{thm:sem} provides the foundation for this evaluation. Following the algorithm for computing the Epstein zeta function in \cite{buchheit2024epstein}, we set $\lambda = 1$ and rescale the lattice so that $V_\Lambda = 1$. Since both the real- and reciprocal-space sums decay at a superexponential rate, they may be truncated to only a small number of lattice points near the origin. A tight cutoff will be derived in forthcoming work.

The Crandall functions appearing in the representation can be evaluated to full precision using the algorithm in \cite{buchheit2024epstein}. Their derivatives, which are also required, can be expressed in terms of Crandall functions, as described in Sec.~\ref{sec:derivatives}. In the reciprocal-space sum, incomplete Bessel functions arise. Their evaluation is addressed in detail in Sec.~\ref{sec:incbessel}.

\subsection{Analytic representation of derivatives}
\label{sec:derivatives}

All derivatives of the Crandall functions can be computed analytically in terms of other Crandall functions.

\begin{lemma}[Derivatives of Crandall functions]
\label{der-crandall}
Let $\nu\in \mathds C$, $\bm \alpha \in \mathds N^d$ , and $\bm r,\bm k\in \mathds R^d\setminus \{\bm 0\}$. Then
\begin{align*}
G_\nu^{(\bm \alpha)}(\bm r) &= 
\sum_{|\bm\beta|\le|\bm\alpha|/2} p_{\bm \alpha,\bm \beta}(\bm r)
G_{\nu+2|\bm\alpha|-2|\bm\beta|}( \bm r),\\
g_\nu^{(\bm \alpha)}(\bm r) &= 
\sum_{|\bm\beta|\le|\bm\alpha|/2} p_{\bm \alpha,\bm \beta}(\bm r)
g_{\nu+2|\bm\alpha|-2|\bm\beta|}( \bm r),\\
G_{\nu}^{(\bm \alpha)}( \bm k, \bm r)&= \sum_{|\bm\beta|\le|\bm\alpha|/2} p_{\bm \alpha,\bm \beta}(\bm r)G_{\nu-2|\bm\alpha|+2|\bm\beta|}(\bm k,\bm r),
\end{align*}
where the derivative always acts on $\bm r$ and where $p_{\bm \alpha,\bm \beta}$ is the following monomial,
\[
p_{\bm \alpha,\bm \beta}(\bm r)=(-\pi)^{|\bm \alpha|-|\bm \beta|} \binom{\bm \alpha}{\bm \beta}  \frac{(\bm \alpha-\bm \beta)!}{(\bm \alpha-2\bm \beta)!}
(2\bm r)^{\bm\alpha-2 \bm \beta}.
\]
Furthermore, the derivatives of the Crandall functions extend continuously to $\bm r=\bm 0$ with
\begin{align*}
\label{eq:derG}
G_{\nu}^{(\bm \alpha)}(\bm 0)&= 
-2\chi(\bm\alpha)\frac{p_{\bm\alpha,\bm\alpha/2}(\bm 0)}{\nu+|\bm\alpha|} 
,&&\mathrm{Re}(\nu)+\vert  \bm \alpha\vert<0, 
\\
g_{\nu}^{(\bm \alpha)}(\bm 0)&= 
2\chi(\bm\alpha)\frac{p_{\bm\alpha,\bm\alpha/2}(\bm 0)}{\nu+|\bm\alpha|} 
,&&\mathrm{Re}(\nu)+\vert  \bm \alpha\vert>0,
\end{align*}
where the right-hand sides form the meromorphic continuation to $\nu \in \mathds C\setminus \{-|\bm \alpha|\}$.
Here, $\chi(\bm\alpha)=1$ if every entry of $\bm\alpha$ is even and $\chi(\bm\alpha)=0$ otherwise. 
Finally, in the case of $\Re\nu-|\bm\alpha|<0$, the derivative of the incomplete \new{Bessel} function extends continuously to $\bm r=\bm 0$ and $\bm k=\bm 0$ with
$$
G_{\nu}^{(\bm \alpha)}(\bm 0,\bm 0)= 
-2\chi(\bm\alpha)\frac{p_{\bm\alpha,\bm\alpha/2}(\bm 0)}{\nu-|\bm\alpha|} 
$$
where the right-hand side forms the meromorphic continuation to $\nu \in \mathds C\setminus \{|\bm \alpha|\}$.
\end{lemma}

\begin{proof}
\new{We begin with the integral representations of the upper Crandall function,}
\new{\[
G_{\nu}(\bm r) = \,2
\int_{0}^{1}t^{-\nu}e^{-\pi \bm r^2/t^2}\, \frac{\mathrm d t}{t}.
\]}
\new{As the Gaussian tensorizes, we may compute the derivative of each one-dimensional factor separately, obtaining}
\new{\[
\frac{\partial^n}{\partial x^n} e^{-\pi x^2/t^2}=(-1)^n (\sqrt{\pi}/t)^n H_n\Big(\sqrt{\pi}x/t\Big) e^{-\pi x^2/t^2} ,
\]}
\new{with $H_n$ the Hermite polynomials. Using the explicit representation of the Hermite polynomials \cite[22.3.10]{abramowitz1948handbook},
\[(-1)^n H_n(x) = n!\sum_{\ell=0}^{\lfloor n/2\rfloor} \frac{(-1)^{n-\ell} (2x)^{n-2\ell}}{\ell! (n-2\ell)!} =   \sum_{\ell=0}^{\lfloor n/2\rfloor} (-1)^{n-\ell} \binom{n}{\ell}\frac{(n-\ell)!}{(n-2\ell)!}(2x)^{n-2\ell}.\]}
\new{After including the definition of $p_{n,\ell}$, we find that
\[\frac{\partial^n}{\partial x^n} e^{-\pi x^2/t^2}= \sum_{n=0}^\ell t^{-2(n-\ell)} p_{n,\ell}(x)  e^{-\pi x^2/t^2}.
\]}\new{The formula for the derivative for the upper Crandall function is then readily obtained after tensorization. The proof of the lower Crandall function proceeds analogously after replacing the integration boundaries.}
For the incomplete Bessel function, inserting the integral representation yields
$$
G_{\nu}^{(\bm \alpha)}(\bm k,\bm r)=
\sum_{|\bm\beta|\le|\bm\alpha|/2}p_{\bm\alpha,\bm\beta}(\bm r)
\,2
\int_{0}^{1}t^{-\nu+2|\bm\alpha|-2|\bm\beta|}e^{-\pi \bm k^2/t^2} e^{-\pi \bm r^2 t^2}\, \frac{\mathrm d t}{t}
$$
and in the case of $\Re\nu-|\bm\alpha|<0$, $\chi(\bm a)=1$ and $\bm \beta=\bm\alpha/2$ we obtain
$$\lim_{h\to 0_+}G_{\nu-|\bm\alpha|}(h\bm k,h\bm r)=
2\int_0^1 t^{-(\nu-|\bm\alpha|)-1}\mathrm{d}t=-\frac{2}{\nu-|\bm\alpha|}.
$$
By dominant convergence
and by the identity theorem, $
-2\chi(\bm\alpha)p_{\bm\alpha,\bm\alpha/2}(\bm 0)/(\nu-|\bm\alpha|)
$
provides the unique holomorphic continuation of $G_{\nu}^{(\bm \alpha)}(\bm 0,\bm 0)$ to $\nu\in\mathds C\setminus\{|\bm\alpha|\}$.
\end{proof}

\begin{remark}
The multinomial coefficient $\binom{\bm\alpha}{\bm\beta}$ found in the polynomial $p_{\bm\alpha,\bm\beta}(\bm \cdot)$ in \Cref{der-crandall} can be efficiently evaluated as the product of $d$ one-dimensional binomial coefficients, see \cite{goetgheluckComputingBinomialCoefficients1987}
and the fraction
$$
\frac{(\bm\alpha-\bm\beta)!}{(\bm\alpha-2\bm\beta)!}
$$
can \new{be} obtained by multiplication of at most $\lfloor\bm\alpha/2\rfloor$ numbers.   
\end{remark}

\begin{algorithm}
\DontPrintSemicolon
\caption{Computation of the incomplete Bessel function.}
\KwIn{ $\nu\in\mathds{R}$, $\bm{k},\bm{r}\in\mathds{R}^d$}

    \BlankLine
    $x=\pi \bm k^2;\quad y=\pi \bm r^2;\quad s =-\nu/2$  
    \BlankLine
    \tcc{Vanishing arguments}
    \If{$x+y=0$}{
        \Return{$1/s$}
    }

    \BlankLine
    \tcc{Vanishing first arguments}
    \If{$x = 0$}{
        \Return{$g_{-\nu}(\bm r)$}  % $\gamma(s,y)/y^s$
    }
    \BlankLine
    
    \tcc{Vanishing second arguments}
    \If{$y=0$}{
        \Return{$G_{\nu}(\bm k)$}% $=x^s\Gamma(s,x)$
    }
    \BlankLine

   \tcc{Vanishing upper bound}
     \If{$2 \;\!\sqrt{\pi} \;\!x^{-(\nu+1)/4}  y^{(\nu-1)/4}e^{(\nu^2/16 - 2  x  y)/\sqrt{x y} }
    <10^{-16}
    $}{
   \Return{$0$}
   }
    
    \BlankLine
    \tcc{Reflect parameters for upper half of the plane}
    \uIf{$x +1/5< y$}{

          $\texttt{swap}=\texttt{True};\quad (s,x,y)\leftarrow(-s,y,x)$

    }
    \Else{

        $\texttt{swap}=\texttt{False}$
        
    }

    \BlankLine
    \tcc{Choose series expansion/recursive algorithm}
    \uIf{$x+y<3/2$}{
          
       $\texttt{result} = G_{-2s}(\bm k)$
       
        \For{$j=1,\ldots,20$}{
            $\texttt{result} \pluseq  G_{-2(s+j)}(\bm k)(-y)^j/j!$
        }
    }
    \Else{

    $n_1=0;\quad n_2=0;\quad n_3=1$   
    
    $d_1=0;\quad d_2=e^{x+y};\quad d_3=(x-y+s+1) d_2$

    \For{$j=2,\ldots,100$}{
        $N=((x -y + s + 1 + 2 (j - 1) ) n_3 + (2 y - 
      s - (j - 1)) n_2 - y n_1)/j$
      
      $D=((x -y + s + 1 + 2 (j - 1) ) d_3 + ((2 y - 
      s - (j - 1)) d_2 - y d_1)/j$

      $n_1=n_2;\quad n_2=n_3;\quad n_3=N$

      $d_1=d_2;\quad d_2=d_3;\quad d_3=D$
    
    }   
    $\texttt{result}=N/D$
}
\BlankLine

    \tcc{Reflect result for upper half of the plane}
    \If{$\textnormal{\texttt{swap}}$}{
        $\texttt{result} \leftarrow 2 (x/y)^{s/2}K_{s}(2\sqrt{x y})-\texttt{result}$
    } 

    \BlankLine
     \Return \texttt{result}
     \label{alg}
\end{algorithm}

\subsection{Computation of the incomplete Bessel function}
\label{sec:incbessel}

In this section, we discuss how the incomplete Bessel function in \Cref{thm:crandall} can be evaluated to full precision.
The algorithm is mainly based on the series expansion in
\cite{harrisIncompleteBesselGeneralized2008} and the recursive algorithm in \cite{slevinskyRecursiveAlgorithmEfficient2023}.
In the following, we focus on $\nu\in\mathds R$.

An algorithm for the computation of the incomplete Bessel function is presented in
\hyperref[alg]{Algorithm 1}.
The different evaluation regions of the \hyperref[alg]{Algorithm 1} are shown
in \Cref{fig:regions}.
In what follows, we discuss the procedures employed in \hyperref[alg]{Algorithm 1} and show the error is flat over the whole parameter range.

For at least one vanishing vector argument, the incomplete Bessel function reduces to the upper and lower Crandall functions, readily available in EpsteinLib.

\begin{lemma}
\label{zero}
Let $\nu\in\mathds R\setminus\{0\}$, then
$$
G_{\nu}(\bm 0,\bm 0)=-\frac{2}{\nu}.
$$
Let $\nu\in\mathds R$ and
$\bm k\in\mathds R^d$, then 
$$
G_{\nu}(\bm k,\bm 0)=G_{\nu}(\bm k).
$$
Let $\bm r\in\mathds R^d$ where $\bm r\neq \bm 0$ if $\nu\in2\mathds N$, then 
$$
G_{\nu}(\bm 0,\bm r)=g_{-\nu}(\bm r).
$$
\end{lemma}

\begin{proof}
The integral representation of $G_{\nu}(\bm k,\bm 0)$
reduces to the integral representation of $G_{\nu}(\bm k)$ in \cite[Lemma B.1.]{buchheit2024epstein},
yielding in particular
$G_{\nu}(\bm 0,\bm 0)=G_{\nu}(\bm 0)=-2/\nu$.
Furthermore, $G_{\nu}(\bm 0,\bm r)$
reduces to the second integral representation of $g_{-\nu}(\bm r/\lambda)$ for $\lambda =1$,
as shown in the proof \cite[Lemma 2.10]{buchheit2024epstein}.
\end{proof}

We may always swap non-zero vector arguments at the cost of computing a modified Bessel function of the second kind, as the following lemma shows. 
This allows us to restrict our computation to incomplete Bessel functions where $\bm k^2+C > \bm r^2$ for some $C>0$,
where the algorithm achieves superior numerical stability.

\begin{lemma}[Incomplete Bessel reflection formula]
\label{bessel-reflection}
Let $\nu\in\mathds R$, and 
let $\bm k,\bm r\in\mathds R^d\setminus\{\bm 0\}$. Then
$$
    G_{\nu}(\bm k,\bm r)   = 
    2(\bm r^2/\bm k^2)^{\nu/4}K_{\nu/2}(2\pi\sqrt{\bm k^2\bm r^2})-G_{-\nu}(\bm r, \bm k)
$$
where $K_{\nu}(2z)$, $s\in\mathds R$, $z>0$ is the modified Bessel function of the second kind
$$
K_{\nu}(2z)=\frac{z^{\nu}}{2}\int_0^{\infty}t^{-\nu}e^{-t-z^2/t}\frac{\rm d t}{t}.
$$
\end{lemma}

\begin{proof}
Substitute $s=1/t$ in the integral representation of the incomplete Bessel function in \Cref{thm:crandall} to obtain
$$
G_{\nu}(\bm k,\bm r)
=\int_0^{\infty}2\;\! t^{\nu}e^{-\pi \bm k^2 t^2}e^{-\pi \bm r^2/t^2}\, \frac{\mathrm d t}{t}
-G_{-\nu}(\bm r,\bm k),
$$
where we have extended the integration domain to $(0,\infty)$,
by adding and subtracting the integral from $0$ to $1$.
After setting $u=\pi \bm k^2 t^2$, 
the integral above takes the form
$$
\bigg(\frac{1}{\pi \bm k^2}\bigg)^{\nu/2}\int_0^{\infty} u^{\nu/2} 
e^{-u-\pi^2\bm k^2\bm r^2/u}
\, \frac{\mathrm d u}{u}
=2(\bm r^2/\bm k^2)^{\nu/4}K_{-\nu/2}(2 \pi \sqrt{\bm k^2\bm r^2})
$$
and the statement follows from the symmetry $K_{s}(z)=K_{-s}(z)$  
\cite[Sec. 7.2.2, Eq. (14)]{Batemann1953b}.
\end{proof}

For small vector arguments, we may now use the series expansion
in the second variable \cite[Eq. (16)]{harrisIncompleteBesselGeneralized2008}.
For intermediate vector arguments, we employ the recursive algorithm due to \cite{slevinskyRecursiveAlgorithmEfficient2023}.
For large vector arguments, we may assume the incomplete Bessel function to be zero, if the following upper bound is smaller than $10^{-16}$.

\begin{lemma}
Let $\nu\in\mathds R$, and 
let $\bm k,\bm r\in\mathds R^d\setminus\{\bm 0\}$. Then
$$
 G_{\nu}(\bm k,\bm r)
\le
 2(\bm k^2)^{-(\nu+1)/4} (\bm r^2)^{(\nu-1)/4}
\exp\bigg(\frac{\nu^2}{16 \pi\sqrt{\bm k^2\bm r^2}}-2 \pi \sqrt{\bm k^2\bm r^2}\bigg).
$$
\end{lemma}

\begin{proof}
Following the proof of \Cref{bessel-reflection} we have
$$
G_{\nu}(\bm k,\bm r)
\le
2(\bm r^2/\bm k^2)^{\nu/4}K_{s}(2z)
$$
for $s=\nu/2$, $z=\pi \sqrt{\bm k^2\bm r^2}$.
An upper bound to the modified Bessel function is readily obtained from the integral representation
 \cite[Sec. 17.12, Eq. (21)]{Batemann1953b},
$$
K_s(2z)=\int_0^{\infty}\cosh(s t)e^{-2z\cosh t}\d t.
$$
As $1+s^2/2\le \cosh s\le e^{|s|}$, we find that
$$
K_s(2z)
\le e^{-2z}\int_0^{\infty}e^{-z (t^2-|s|t/z)}\d t.
$$
After completing the squares $t^2-|s|t/z=(t-|s|/(2z))^2-s^2/(4z^2)$
and extending the integration range to $(-\infty,\infty)$
we obtain
$$
K_s(2z)
\le e^{s^2/(4z)-2z}\int_{-\infty}^{\infty}e^{-z (t-|s|/(2z))^2}\d t
$$
where the Gaussian integral evaluates to $\sqrt{\pi/z}$.
\end{proof}

\begin{figure} 
    \centering
\includegraphics[width=.55\linewidth]{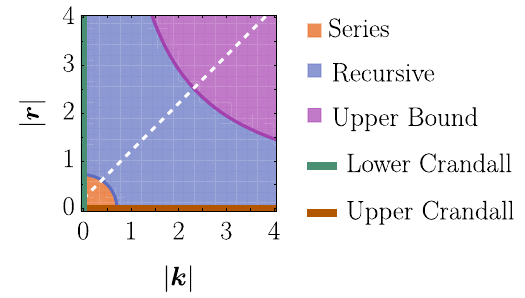}
    \caption{Regions of \hyperref[alg]{Algorithm 1}.
    The white line indicates the split, where above the dashed white line we employ the reciprocal values.
    The purple region where the upper bound of the incomplete Bessel function is smaller than $10^{-16}$ is shown for $\nu = 0$. 
    }
\label{fig:regions}
\end{figure}

\begin{table}
\centering
\begin{tabular}{ccccc}
    \toprule
    $\nu$ 
    & \(E_{\mathrm{med}}\) 
    & \(E_{\mathrm{max}}\) 
    & $t_{\mathrm{med}}\,[s]$
    & $t_{\mathrm{max}}\,[s]$
  \\

\midrule
	$-4$ 
	& $1.18 \cdot 10^{-17}$ 
	& $4.91 \cdot 10^{-15}$ 
	& $1 \cdot 10^{-6}$ 
	& $1 \cdot 10^{-5}$
\\

\midrule
	$-2$ 
	& $1.24 \cdot 10^{-17}$ 
	& $3.14 \cdot 10^{-15}$ 
	& $1 \cdot 10^{-6}$ 
	& $5 \cdot 10^{-6}$
\\

\midrule
	$0$ 
	& $1.26 \cdot 10^{-17}$ 
	& $4.88 \cdot 10^{-15}$ 
	& $1 \cdot 10^{-6}$ 
	& $8 \cdot 10^{-6}$
\\

\midrule
	$2$ 
	& $1.34 \cdot 10^{-17}$ 
	& $4.55 \cdot 10^{-15}$ 
	& $1 \cdot 10^{-6}$ 
	& $5 \cdot 10^{-6}$
\\

    \bottomrule
\end{tabular}
\caption{
Minimum of absolute and relative error of $G_{\nu}(\bm r,\bm k)$
    obtained by \hyperref[alg]{Algorithm 1} for $\nu=\{-4,-2,0,2\}$.
    Every value corresponds to the a grid of
    $1/10\le|\bm r|,|\bm k|\le 4$ \new{with} $1/10$.
    Reference values obtained by high-precision integral evaluations.
}
\label{fig:algacc}
\end{table}

We now assess the precision and runtime of incomplete Bessel function evaluations by
 \hyperref[alg]{Algorithm 1}.
Let $E$ denote the minimum of the absolute and relative error
$$
E=\min \{E_{\rm{abs}},E_{\rm rel}\}
$$
of our algorithm in comparison with high-precision values obtained by the integral representation.
In \Cref{fig:algacc}
we show the median and maximal error $E_{\rm med}$ and $E_{\rm max}$
as well as the median and maximum evaluation times $t_{\rm med}$ and $t_{\rm max}$
over the parameter grid $0\le |\bm k|,|\bm r|\le 4$ of width $1/20$
for some fixed $\nu\in\{-4,-2,0,2\}$.
Note, that we skip the singular cases for $\bm k =\bm r=\bm 0$, $\nu=0$ and
$\bm k = \bm 0$, $\bm r\neq \bm 0$, $\nu\in 2\mathds N$.
We achieve machine precision $E<5\cdot 10^{-15}$ over the whole parameter range
for evaluation times of $t\le 10^{-5}$ seconds.
The values were obtained on an Intel Core i7-1260P (12th Gen) 16-core processor.

\section{Outlook and discussion}
\label{sec:outlook}

In this work, we have presented an efficient and precise method for evaluating infinite lattice sums appearing in micromagnetics with periodic boundary conditions. At the moment of writing of this work, a common practice was to truncate the arising sums, leading to an uncontrolled source of error. We show that the error due to truncation can be removed by an additional efficiently computable correction term that follows from a generalization of the Riemann zeta function to sums over $n$-dimensional lattices in $d$-dimensional space. We present an efficient way to compute the arising zeta functions, including their derivatives, and implement the resulting algorithm in a high-performance C~framework. Benchmarks of the algorithm against known formulas and direct summation for large interaction exponents show that full precision is reached, while the numerical effort for the correction term is much smaller than the effort for typical truncated sums. This makes it possible to reach machine precision for the infinite sum at effectively the same runtime as truncated sums in current software packages. \new{We  emphasize that the numerical method developed in this article is designed for high-precision, single-shot evaluations of the interaction potential between two domains. In micromagnetic simulations, the mutual potentials between all domains need to be computed only once and are subsequently used in the main iterative algorithm.}
\new{Thus, our recipe can be considered as a pre-routine, while the main iterative algorithm can be implemented as usual by using FFT-accelerated convolution and massively parallel frameworks such as OpenMP, CUDA, MPI, etc. Moreover, our pre-routine is in turn straightforwardly parallelizable. For example, for a film geometry as in Fig.~\ref{fig:boxes}(b), the possible parallelism is ${N_x\times N_y\times N_z}$ threads.
At the same time, non-uniform shapes (for example, cavities or non-planar interfaces) are possible, as well as combinations of different materials. 
This kind of non-uniformity does not affect our pre-routine at all, while the main iterative algorithm can be equipped with an array of individual saturation magnetizations for each of the cuboids (for example, a cavity-related cuboid shall correspond to the value zero in this array).
}

The implications of this work are twofold. 
First, quantitative errors due to truncation in micromagnetic systems with periodic boundary conditions can be removed at a negligible numerical cost, yielding reliable results. 
Second, since our method is general and applies to arbitrary interaction exponents, source and target geometries, and lattices, it can be applied to other problems.
 In order to be specific, we briefly describe a few examples below.

\textit{Ferroelectrics}.
Microscopic textures in ferroelectrics are essentially the spatial distribution of the polarization vector, $\mathbf{P}(\mathbf{r})$, within the sample. These textures are very diverse and are the subject of intensive research~\cite{Nataf2020, LUKYANCHUK20251}.
If there are no free electric charges in the sample, i.e. the charge distribution is described solely by polarization, then the application of our method is straightforward due to electromagnetic duality.
That is, instead of the magnetization vector~$\mathbf{M}$ (see \ref{app:cuboids_classical}) we assume the polarization vector~$\mathbf{P}$, plus the corresponding replacement for vector potentials, $\mathbf{A}\mapsto\mathbf{F}$ (electric vector potential), and fields, $\mathbf{B}\mapsto\mathbf{E}$ (electric field).

In cases where $\Omega$ is a point source, there is no need for integration and derivatives for the potential of Eq.~(\ref{U}) as the problem is reduced to a discrete one, while the method we propose here is anyway, a generalization of the previous results~\cite{buchheit2022efficient,buchheit2025computation,buchheit2023exact} where ${n=d}$ to cases ${n<d}$. 
The following examples may apply to these new cases.

\textit{Atomistic spin dynamics}.
Classical spin models~\cite{AtomisticSpinDynamics, RevModPhys.95.035004} are related to micromagnetics within certain limits, while in the general case, these are different problems. 
Spin-spin interactions can exhibit slowly decaying interactions in addition to the dipole–dipole interaction, such as the Ruderman-Kittel-Kasuya-Yosida (RKKY) interaction. 
Accordingly, our new method can be applied to model magnetic multilayers such as van-der-Waals magnets and beyond.

Qualitatively new behavior can be studied in systems where different phases exhibit minute energy differences, i.e. in frustrated systems. 
Similar phenomena have recently been observed in a new Shastry-Sutherland Ising compound and have been theoretically explained using computational techniques that are precursors to this work~\cite{yadav2024observation}.

\textit{Molecular dynamics}.
Periodic boundary conditions are a natural approach in molecular dynamics for interpreting unbounded dimensions. The case $n<d$ is of particular interest, since it is well-known that confinement affects phase transitions, e.g. in the ultra-thin water films~\cite{PhysRevLett.102.050603}. 
In turn, the functions of intermolecular potentials can have rather complex and slowly decaying profiles of the Lennard-Jones type and others~\cite{Robles_Navarro_2025}. 
Accordingly, our new method can significantly improve numerical techniques for studying such systems. 

As a final remark, we emphasize that the high-precision values presented in Table~\ref{tab:dz} can be regarded as fundamental constants of micromagnetics and, in particular, serve as reliable references for benchmarking numerical methods.

\vspace*{.4cm}
{\footnotesize{\noindent\textbf{Acknowledgements}
A.B. is grateful for the support and hospitality of the Pauli Center for Theoretical Studies at ETH Zürich.
}}

\vspace*{.4cm}
{\footnotesize{\noindent\textbf{Funding} 
F.N.R. acknowledges support from the Swedish Research Council (Grant No. 2023-04899).
T.K. acknowledges funding received from the European Union’s Horizon 2020 research and innovation programm under the Marie Skłodowska–Curie grant agreement No 899987.
J.B. thanks the DLR Quantum Computing Initiative for funding his PhD research. This work was supported by the Klaus-Tschira Stiftung under Grant No. 00.025.2025.
}}

\vspace*{.4cm}
{\footnotesize{\noindent\textbf{Data Availability} The authors declare that all data supporting the findings of this work are available within this article. The open-source library EpsteinLib is publicly available, see \cite{buchheit2024epstein}.
The implementation of the incomplete Bessel function is available on 
\href{https://github.com/JoKaBus/incomplete-bessel}{GitHub} 
(\href{https://github.com/JoKaBus/incomplete-bessel}{github.com/JoKaBus/incomplete-bessel})
and includes
a notebook \texttt{examples/mathematica/MicromagneticsReplicated.wls} that reproduces the numerical results presented in this article.
}}

\vspace*{.4cm}
{\footnotesize{\noindent\textbf{Conflict of interest}
The authors declare that they have no conflict of interest.}
}

\vspace*{.4cm}

{\footnotesize{\noindent\textbf{CRediT authorship contribution statement}} \\ \noindent \textbf{Andreas A. Buchheit} Supervision, Writing - Original Draft, Writing - Review \& Editing, Software, Visualization; \textbf{Jonathan K. Busse} Writing - Original Draft, Writing - Review \& Editing, Software, Visualization; \textbf{Torsten Keßler} Writing - Original Draft, Writing - Review \& Editing, Visualization; \textbf{Filipp N.~Rybakov} Writing - Original Draft, Writing - Review \& Editing, Visualization
}

\vspace*{.4cm}
{\footnotesize{\noindent\textbf{Declaration of generative AI and AI-assisted technologies in the writing process}} During the preparation of this work, the authors used ChatGPT and Claude for grammar checking and minor language enhancements. After using this tool, the authors reviewed and edited the content as needed and take full responsibility for the content of the published article.}

\printbibliography

\appendix

\renewcommand{\theequation}{\Alph{section}\arabic{equation}} 

\section{Homogeneous cuboid in micromagnetics}
\label{app:cuboids_classical}
\setcounter{equation}{0}

In this section, we present explicit expressions for the magnetic fields generated by a uniformly magnetized cuboid (rectangular prism). 
We also provide formulas for the demagnetization factors and the interaction energy of two cuboids.
To the best of our knowledge, previous approaches to this problem in the literature were based on rather cumbersome analytical calculations~\cite{RhodesRowlands1954, Schabes_Aharoni_1987, Newell_1993, Hubert, Aharoni_1998}.
In contrast, here we propose a compact approach based on the theory of the potential~\cite{MacMillan1930} and Hertz potentials in electromagnetism~\cite{Essex_about_Hertz_1977}.

\begin{figure*}[!ht]
	\centering
	\includegraphics[width=5.4cm]{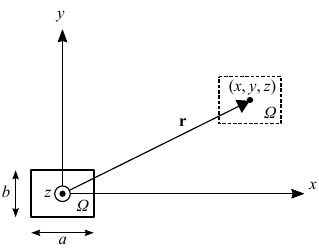}
	\caption{\small
Cuboid $\Omega$ centered in the Cartesian coordinate system~$(x,y,z)$.
}
\label{fig_cuboid}
\end{figure*}

We consider a cuboid-shaped source (domain~$\Omega$) of size ${a\times b\times c}$ placed at the center of the Cartesian coordinate system, see Fig.~\ref{fig_cuboid}. 
\new{
Note that, with respect to the notation for an axis-aligned cuboid introduced in the main text, here we have $(c_1,c_2,c_3) \equiv (a, b, c)$.
}
We call the cuboid a source, meaning that each of its points is considered as a source of Newtonian potential,~$\sim\tfrac{1}{|\mathbf{r}|}$.
The total potential of a cuboid is obtained by integration and can be represented~\cite{Waldvogel1976} in the following canonical form:
\begin{align}
\Phi(x,y,z)   =
\sum\limits_{i,j,k = \pm 1}
i\,j\,k\,F_1\!\left( x + i\tfrac{a}{2}, y + j\tfrac{b}{2}, z + k\tfrac{c}{2}    \right),
\end{align}
with the auxiliary function  
\begin{align}
F_1(u,v,w) = &
-\frac{u^2}{2}\mathrm{arctan}\left(\frac{v\,w}{u\,R}\right) 
-\frac{v^2}{2}\mathrm{arctan}\left(\frac{w\,u}{v\,R}\right) 
-\frac{w^2}{2}\mathrm{arctan}\left(\frac{u\,v}{w\,R}\right) + \nonumber\\
& +  
v\,w\,\mathrm{artanh}\left(\frac{u}{R}\right)+
w\,u\,\mathrm{artanh}\left(\frac{v}{R}\right)+
u\,v\,\mathrm{artanh}\left(\frac{w}{R}\right), \nonumber
\end{align}
and $R \equiv \sqrt{u^2 + v^2 + w^2}$.
Using $\Phi$, we can explicitly write the expression for the Hertz magnetic vector~\cite{Essex_about_Hertz_1977} of a uniformly magnetized cuboid:
\begin{align}
\mathbf{Z}(\mathbf{r}) = \frac{\mu_0}{4\pi}\,\mathbf{M}\,\Phi(\mathbf{r}),
\end{align}
where ${\mu_0 = 4\pi\times 10^{-7}}$~N/A$^2$, and $\mathbf{M}\equiv(M_x,M_y,M_z)$ is the magnetization vector~(in~A/m).
Magnetic vector potential and magnetic field are derived quantities:
\begin{align}
\mathbf{A}(\mathbf{r}) &\equiv \nabla\times\mathbf{Z}(\mathbf{r}),\\
\mathbf{B}(\mathbf{r}) &\equiv \nabla\times\mathbf{A}(\mathbf{r}) = \frac{\mu_0}{4\pi}\, \nabla\times\nabla\times(\mathbf{M}\,\Phi(\mathbf{r})).\label{B} 
\end{align}

We are further interested in the integral result of the action of the magnetic field described by Eq.~(\ref{B}) on some displaced domain of shape~$\Omega$, see Fig.~\ref{fig_cuboid}:
\begin{align}
\mathbf{B}_\text{avg}(\mathbf{r}) = \frac{\int_{\Omega + \mathbf{r}} \mathbf{B} (\bm r')\,d\bm r'}
{\int_{\Omega + \mathbf{r}} d\bm r'} = 
\frac{1}{a\,b\,c} \int_{\Omega + \mathbf{r}} \mathbf{B} (\bm r')\,d\bm r'
.
\end{align}

The components of this vector are linear combinations of the potentials from Eq.~(\ref{U}) for $\nu=1$, $n=0$ and $d=3$ (trivial case for a single cuboid in three-dimensional space), namely,
\begin{align}
&\mathbf{B}_\text{avg}(\mathbf{r}) = \frac{\mu_0}{4\pi\,a\,b\,c} \times \,
\notag \\ &\mathbf{M}
\cdot
\begin{pmatrix}
-U^{(0,2,0)}(\mathbf{r})-U^{(0,0,2)}(\mathbf{r}) & U^{(1,1,0)}(\mathbf{r}) &  U^{(1,0,1)}(\mathbf{r})\\
U^{(1,1,0)}(\mathbf{r}) & -U^{(2,0,0)}(\mathbf{r})-U^{(0,0,2)}(\mathbf{r}) & U^{(0,1,1)}(\mathbf{r})\\
U^{(1,0,1)}(\mathbf{r}) & U^{(0,1,1)}(\mathbf{r}) & -U^{(2,0,0)}(\mathbf{r})-U^{(0,2,0)}(\mathbf{r})
\end{pmatrix}
, \nonumber
\end{align}
and the explicit formulas for the potentials are:
\begin{align}
& U^{(1,1,0)}(\mathbf{r}) = 
\sum\limits_{i,j,k,\alpha,\beta,\gamma = \pm 1}
i\,j\,k\,\alpha\,\beta\,\gamma\,F_2\!\left(z + (k + \gamma)\tfrac{c}{2}, x + (i + \alpha)\tfrac{a}{2}, y + (j + \beta)\tfrac{b}{2}\right), \nonumber \\
& U^{(0,1,1)}(\mathbf{r}) = 
\sum\limits_{i,j,k,\alpha,\beta,\gamma = \pm 1}
i\,j\,k\,\alpha\,\beta\,\gamma\,F_2\!\left( x + (i + \alpha)\tfrac{a}{2}, y + (j + \beta)\tfrac{b}{2}, z + (k + \gamma)\tfrac{c}{2}\right), \nonumber  \\
& U^{(1,0,1)}(\mathbf{r}) = 
\sum\limits_{i,j,k,\alpha,\beta,\gamma = \pm 1}
i\,j\,k\,\alpha\,\beta\,\gamma\,F_2\!\left( y + (j + \beta)\tfrac{b}{2}, z + (k + \gamma)\tfrac{c}{2}, x + (i + \alpha)\tfrac{a}{2}\right), \nonumber  \\
& U^{(2,0,0)}(\mathbf{r}) = 
\sum\limits_{i,j,k,\alpha,\beta,\gamma = \pm 1}
i\,j\,k\,\alpha\,\beta\,\gamma\,F_3\!\left(x + (i + \alpha)\tfrac{a}{2}, y + (j + \beta)\tfrac{b}{2}, z + (k + \gamma)\tfrac{c}{2}\right), \nonumber  \\
& U^{(0,2,0)}(\mathbf{r}) =
\sum\limits_{i,j,k,\alpha,\beta,\gamma = \pm 1}
i\,j\,k\,\alpha\,\beta\,\gamma\,F_3\!\left(y + (j + \beta)\tfrac{b}{2}, z + (k + \gamma)\tfrac{c}{2}, x + (i + \alpha)\tfrac{a}{2}\right), \nonumber  \\
& U^{(0,0,2)}(\mathbf{r}) =
\sum\limits_{i,j,k,\alpha,\beta,\gamma = \pm 1}
i\,j\,k\,\alpha\,\beta\,\gamma\,F_3\!\left(z + (k + \gamma)\tfrac{c}{2}, x + (i + \alpha)\tfrac{a}{2}, y + (j + \beta)\tfrac{b}{2}\right), \nonumber  
\end{align}
where we introduce two additional auxiliary functions:
\begin{align}
F_2(u,v,w) = 
&
-\frac{v w R}{3}
-\frac{u^3}{6}\mathrm{arctan}\left(\frac{v\,w}{u\,R}\right) 
-\frac{u\,v^2}{2}\mathrm{arctan}\left(\frac{u\,w}{v\,R}\right) 
-\frac{u\,w^2}{2}\mathrm{arctan}\left(\frac{u\,v}{w\,R}\right) 
+ \nonumber\\
&  
+ u\,v\,w\,\mathrm{artanh}\left(\frac{u}{R}\right)
+ \frac{3\,u^2\,w - w^3}{6}\mathrm{artanh}\left(\frac{v}{R}\right)
+ \frac{3\,u^2\,v - v^3}{6}\mathrm{artanh}\left(\frac{w}{R}\right) 
,\nonumber
\end{align}
and 
\begin{align}
F_3(u,v,w) = 
&
\frac{3u^2-R^2}{6}R
- u\,v\,w\,\mathrm{arctan}\left(\frac{v\,w}{u\,R}\right) + \nonumber\\
&
+ \frac{v\,(w^2-u^2)}{2}\mathrm{artanh}\left(\frac{v}{R}\right) 
+ \frac{w\,(v^2-u^2)}{2}\mathrm{artanh}\left(\frac{w}{R}\right) 
. \nonumber
\end{align}

Let us denote by $\mathbf{M}_\text{disp}$ the magnetization of the displaced cuboid and then the mutual energy (in~J) of two cuboids:
\begin{align}
W = 
\,a\,b\,c\,
\mathbf{B}_\text{avg}
\cdot
\mathbf{M}_\text{disp}
. 
\end{align}

The interaction of a cuboid with itself is characterized by demagnetizing factors:
\begin{align}
D_x = 1 + \frac{U^{(0,2,0)}(0) + U^{(0,0,2)}(0)}{4\pi\,a\,b\,c}, \nonumber \\
D_y = 1 + \frac{U^{(2,0,0)}(0) + U^{(0,0,2)}(0)}{4\pi\,a\,b\,c}, \label{demagnetizing} \\
D_z = 1 + \frac{U^{(2,0,0)}(0) + U^{(0,2,0)}(0)}{4\pi\,a\,b\,c}. \nonumber
\end{align}
Formulas~(\ref{demagnetizing}) involve the use of only the auxiliary function~$F_3$ and are therefore somewhat simpler than the long formula suggested by Aharoni~\cite{Aharoni_1998}. 
In the special case when a cuboid is a cube, i.e. if $a = b = c$, we obtain the same factors as for the sphere~\cite{Newell_1993}: $D_x=D_y=D_z=\tfrac{1}{3}$.

\section{Integral representation of summand function for cuboids} 
\label{app:cuboids_integral}
\setcounter{equation}{0}

In the case of cuboids as defined in Theorem~\ref{thm:sem}, the double integral for the single summand can be evaluated semi-analytically as follows. The result serves as an alternative to formulas derived in \ref{app:cuboids_classical} for $\nu = 1$ and follows directly from the integral representation of the Crandall functions. The single summand for the cuboid $\Omega=\prod_{\ell=1}^d [-c_\ell/2,c_\ell/2]$ in Theorem~\ref{thm:sem} admits the representation
\begin{align*}
&S^{(\bm m)}(\bm r)= 2 c_1^2 c_2^2\dots c_d^2 \,\int_{0}^\infty t^{-\nu} \prod_{\ell=1}^d\psi_{\alpha_\ell}( r_\ell, t,c_\ell)  \frac{\mathrm d t}{t},
\end{align*}
where the integral can be meromorphically continued to $\nu\in \mathds C$ by means of the Hadamard integral \cite{gelfand1964generalizedI,buchheit2023exact}.
The functions $\psi_{m}$ are given by
\[
\psi_m (r,t,c)=\frac{1}{c^2}\int_{-c/2}^{c/2} \partial_{r'}^{m} \int_{-c/2}^{c/2}e^{-\pi (r+r'-r'')/t^2}\,\mathrm d r'' \,\mathrm d r',\quad m\in \{0,1,2\}.
\]
and the double integral can be evaluated as follows
    \begin{align*}
        \psi_m(r,t,c) &= t^{-m} (t/c)^{2}  \Big(f^{(m)}\big((r-c)/t\big) - 2f^{(m)}(r/t) +f^{(m)}\big((r+c)/t\big),
        \end{align*}
        with the function $f$ defined as 
        \begin{align*}
        f(r) &= e^{-\pi r^2}/(2\pi)  + r\,\mathrm{erf}(\sqrt{\pi} r)/2,
    \end{align*}
    and its derivatives
    \[
        f'(r) = \mathrm{erf}(\sqrt{\pi} r)/2,\quad 
        f''(r) = e^{-\pi r^2},
    \]
    with $\mathrm{erf}$ denoting the error function.

\newpage

\section{Glossary of symbols and terms}
\label{app:glossary}
\setcounter{equation}{0}

\begin{table}[h!]
\caption{\label{table_glossary} Table of symbols and terms.}
\begin{tabular}{p{18mm} p{85mm} r }
\multicolumn{1}{l}{Symbol} & \multicolumn{1}{c}{Meaning}  \\
\hline
$\subset$ & subset of, used for finite or compact subsets  \\
$\subseteq$  & subset of or equal to, used for possibly unbounded sets \\
$\backslash$  & set-theoretic difference \\
$\{ a \}$ & set consisting of a single element~$a$  \\
$!$  & factorial \\
$\sup$ & supremum \\
$\mathrm{dist}$ & distance \\
$\mathrm{diam}$ & diameter \\
$\Vol$ & volume \\
$\|\bm r\|_{\infty}$ & infinity norm = maximum absolute value  \\
$\mathds R^d$ & $d$-dimensional Euclidean space  \\
$\mathds R_+$ &  positive real numbers \\
$\mathds N$ & natural numbers including zero  \\
$\mathds N^n$ & $n$-tuple of numbers from $\mathds N$  \\
$\mathds N_+$ & positive natural numbers   \\
$\mathds Z$ & integers  \\
$\mathds Z^d$ & $d$-tuple of numbers from $\mathds Z$ \\
$\mathds C$ & complex numbers  \\
$\nabla$ & gradient \\
$\Lambda$ & $n$-dimensional Bravais lattice \\
\hline
\end{tabular}
\end{table}

\end{document}